\documentclass{article}

\usepackage{amsfonts}
\usepackage{amsthm}
\usepackage{amsmath}
\usepackage{amssymb}
\usepackage[all]{xy}
\usepackage{epic,eepic}
\usepackage{bm}
\usepackage[T1]{fontenc}

\newcommand{\C}{\mathbb C}
\newcommand{\R}{\mathbb R}
\newcommand{\N}{\mathbb N}

\newcommand{\Z}{\mathbb Z}

\newcommand{\eps}{\varepsilon}

\newcommand{\norm}[1]{\left\Vert #1\right\Vert}

\newcommand{\abs}[1]{\left\vert #1 \right\vert}
\newcommand{\seq}[2]{(#1_{#2})_{#2\in\N}}

\newcommand{\sumo}[2]{\sum_{#1=1}^#2}

\newcommand{\ball}{\mathrm{ball}}

\newcommand{\up}{\mathrm}

 \newcommand{\lopen}{\mathopen{]}}
 \newcommand{\ropen}{{\mathclose{[}}}
\newtheorem{thm}{Theorem}[section]

\newtheorem{lem}[thm]{Lemma}
\newtheorem{prop}[thm]{Proposition}

\newtheorem{corollary}[thm]{Corollary}

\newtheorem{Theorem}[thm]{Theorem}

\theoremstyle{definition}

\theoremstyle{definition}

\theoremstyle{definition}

\newtheorem{ex}[thm]{Example}

\begin{document}


\title{Compactness of derivations from \\commutative  Banach algebras.}


\author{Matthew J. Heath}

\maketitle

\begin{abstract}
We consider the compactness of derivations from commutative Banach algebras into their dual modules. We show that if there are no compact derivations from a commutative Banach algebra, $A$, into its dual module, then there are no compact derivations from $A$ into any symmetric $A$-bimodule; we also prove analogous results for weakly compact derivations and for bounded derivations of finite rank. We then characterise the compact derivations from the convolution algebra $\ell^1(\Z_+)$ to its dual. Finally, we give an example (due to J. F. Feinstein) of a non-compact, bounded derivation from a uniform algebra $A$ into a symmetric $A$-bimodule.  
\end{abstract}
\section{Introduction}The question of the compactness of endomorphisms of Banach algebras has been studied in, for example \cite{Kamowitzcompact}, \cite{FeiKamHinfty}, and \cite{FeiKamdiferentiable}.
In this paper we consider compactness for another class of maps of interest in Banach algebra theory, derivations from a Banach algebra to its dual.  
In \cite{ChoiHeathAD} Yemon Choi and the present author  showed that all derivations from the disc algebra to its dual are compact. In \cite{ChoiHeathell1} the same two authors characterised when derivations from $\ell^1(\Z_+)$ to its dual are weakly compact.
\subsection{Definitions and notation}
Throughout we shall take all Banach spaces to be over the field of complex numbers.

Let $A$ be a Banach algebra. Recall that a \emph{Banach $A$-bimodule} is a Banach space together with two bilinear maps $A\times E\rightarrow E$ denoted $(a,x)\mapsto a\cdot x$ and $(a,x)\mapsto x\cdot a$ such that:
\[a\cdot(b\cdot x)+(ab)\cdot x,\,(x\cdot a)\cdot b=x\cdot (ab),\,a\cdot(x\cdot b)=(a\cdot x)\cdot b\qquad(a,b\in A,x\in E).\] 
Clearly, if we define both actions to be the product, $A$ becomes an $A$-bimodule. If $E$ and $F$ are Banach $A$-bimodules we call a linear map $R:E\rightarrow F$ an \emph{$A$-bimodule} homomorphism if:
\[R(a\cdot e)=a\cdot R(e),\,R(e\cdot a)=R(e)\cdot a\qquad(a\in A, e\in E).\]

We say $E$ is a \emph{symmetric} $A$-bimodule if, for all $a\in A$ and all $x\in E$, we have $a\cdot x=x\cdot a$.  

Let $A$ be an algebra and $E$ an $A$-bimodule. We call a linear map $D:A\rightarrow B$ a \emph{derivation} if the following identity holds for all $x,y\in A$:
\[D(ab)=a\cdot D(b)+D(a)\cdot b.
\]
The derivation $D$ is called \emph{inner} if the is $e\in E$ such that $D(a)=a\cdot e-e\cdot a$ for all $a\in A$. We call this inner derivation $\delta_e$. If $E$ is symmetric, it is clear that the only inner derivation from $A$ into $E$ is the zero derivation. 

For a Banach space $E$ we denote the topological dual of $E$ by $E^*$. If $A$ is a Banach algebra and $E$ a Banach $A$-bimodule we make $E^*$ into a Banach $A$-bimodule by defining the actions:
\[(a\cdot \psi)(x)=\psi (x\cdot a),\, (\psi \cdot a)(x)=\psi(a\cdot y),\qquad (a\in A, \psi\in E^*, x\in E).\]
\section{Results}
\subsection{General results}

Recall the well-known result, due to Bade \emph{et al.} (\cite{BCD}, also found as \cite[2.8.63 (iii)]{Dales}) that, if a commutative Banach algebra $A$ has no non-zero, bounded derivations into $A^*$ (i.e. if $A$ is 
weakly amenable), then it has no non-zero, bounded derivations into any symmetric $A$-bimodule.  In this subsection we prove analogues of this result with ``bounded'' replaced by ``compact'', by ``weakly compact'' and by ``bounded and of rank less than $n$'' for some $n\in\N$.

We shall need the following lemma, which is a stronger version of \cite[2.8.63 (i)]{Dales} 
\begin{lem}\label{rank1}
Let $A$ be a Banach algebra with no non-zero, non-inner, bounded derivations of rank $1$, into $A^*$. Then $\overline{A^2}=A$.
\end{lem}
\begin{proof}We prove this result in the contrapositive. Let $A$ be a Banach algebra such that $\overline{A^2}\ne A$; we shall construct the required derivation. Take $a_0\in A\setminus \overline{A^2}$. By the Hahn-Banach theorem we may choose 
$\lambda_0\in A^*$ with $\lambda_0\vert_{A^2}=0$ and $\lambda_0(a_0)=1$. We define a function as follows:
\begin{eqnarray*}
D:&A&\rightarrow A^*\\
&a&\mapsto\lambda_0(a)\lambda_0.
\end{eqnarray*}
It is clear that $D$ is a bounded linear map. Also, $D(A) = \lambda_0\C$ and so $D$ is of rank $1$. Since $\lambda_0(A^2)=0$, we have, for $a,b,c \in A$,  
\[D(ab)(c)=\lambda_0(ab)\lambda_0(c)=0.\]
and 
\begin{eqnarray*}
(a\cdot D(b)+D(a)\cdot b)(c)&=&D(b)(ca)+D(a)(bc)\\
&=&\lambda_0(b)\lambda_0(ca)+\lambda_0(a)\lambda_0(bc)=0.
\end{eqnarray*} 
Hence, $D(ab)=a\cdot D(b)+D(a)\cdot b=0$ and so 
$D$ is a derivation. Furthermore, 
\[\norm{D(a_0)(a_0)}=\abs{\lambda_0(a_0)}\abs{\lambda_0(a_0)}=1\] 
while, for each $\lambda\in A^*$,
\begin{eqnarray*}
\delta_{\lambda}(a_0)(a_0)&=&(a_0\cdot\lambda)(a_0)-(\lambda\cdot a_0)(a_0)\\
&=&\lambda(a_0^2)-\lambda(a_0^2)=0.
\end{eqnarray*}
Thus $D$ is not inner, and so the result follows. 
\end{proof}
We shall also want the following elementary lemma, which may be found as \cite[2.6.6.(i)]{Dales}.
\begin{lem}\label{Daleslem}
Let $A$ be a commutative Banach algebra, let $E$ be a symmetric $A$-bimodule and let $\lambda\in E^*$. Then there is a bounded $A$-bimodule homomorphism $R_\lambda:E\rightarrow A^*$ such that 
\[R_\lambda(x)(a)=\lambda(a\cdot x)\quad(a\in A, x\in E).\]
\end{lem}
We can now prove the main result of this subsection. Each part of the proof follows the pattern of \cite[2.8.63 (iii)]{Dales}.
\begin{Theorem} Let $A$ be a commutative Banach algebra. Then the following are true:
\begin{enumerate}
\item if $A$ has no non-zero, bounded, derivations of rank less than $n\in\N$ into $A^*$ then it has no non-zero derivations of rank 
less than $n$ into any symmetric Banach $A$-bimodule $E$;
\item if $A$ has no non-zero, compact derivations into $A^*$ then it has no non-zero compact 
derivations into any symmetric Banach $A$-bimodule $E$;
\item if $A$ has  no non-zero, weakly-compact derivations into $A^*$ then it has no non-zero weakly-compact derivations into any 
symmetric Banach $A$-bimodule $E$.
\end{enumerate}

\end{Theorem}
\begin{proof}
In each case we shall assume, towards contradiction, that such a derivation exists. First, let $E$ be a symmetric $A$-bimodule and let 
$D:A\rightarrow E$ be any non-zero, bounded derivation. By 
Theorem \ref{rank1}, $\overline{A^2}=A$, and so there is $a_0\in A$ with 
$D\left(a_0^2\right)\ne 0$. Thus, $a_0\cdot D_{a_0}=1/2D\left(a_0^2\right)\ne 0$, and so, by the Hahn-Banach theorem,
there exists $\lambda_D \in E^*$ such that $\lambda(a_0\cdot D(a_0))=1$. By 
Lemma \ref{Daleslem}  there is a continuous $A$-bimodule homomorphism 
$R_{\lambda_D}:E\rightarrow A^*$ such that 
$R_{\lambda_D} (x)(a)=\lambda_D(a\cdot x)$ for each $x\in E$. 
Now let $D'=R_{\lambda_D}\circ D:A\rightarrow A^*$. Clearly $D$ is a bounded linear map, and since 
$R_{\lambda_D}$ is an $A$-bimodule homomorphism we have, for $a,b\in A$,
\begin{eqnarray*}
D'(ab)&=&R_{\lambda_D}(D(ab))\\
&=&R_{\lambda_D}(a\cdot D(b)+b\cdot D(a))\\
&=&a\cdot R_{\lambda_D}(D(b))+b\cdot R_{\lambda_D}(D(a))\\
&=& a\cdot D'(b)+b\cdot D'(a).
\end{eqnarray*}
Thus, $D'$ is a derivation. Also, $D'(a_0)(a_0)=\lambda_D(a_0\cdot D(a_0))=1$ 
and so $D'\ne0$. 

To show part (1) we now let $n\in\N$ and $D$ be a bounded derivation of rank less than $n$. 
Then $D'(A)=R_{\lambda_D}(D(A))$  is a linear image of a space of dimension less than $n$. Hence, $D'(A)$ has dimension less than $n$ 
and so $R_\lambda \circ D$ has rank less than $n$.

To show part (2) we let $D$ be a compact derivation. Then 
\begin{equation}\label{kitten}
\overline{R_{\lambda_D}\left(\overline{D(\ball(A))}\right)}\supseteq \overline{D'(\ball(A))}=
\overline{R_{\lambda_D}(D(\ball(A)))}\supseteq R_{\lambda_D}\left(\overline{D(\ball(A))}\right).
\end{equation}
Now, since $D$ is compact, $\overline{D(B(A))}$ is compact and so $R_{\lambda_D}\left(\overline{D(B(A))}\right)$ is compact. In particular, $R_{\lambda_D}\left(\overline{D(B(A))}\right)$ is closed. Thus, (\ref{kitten}) gives
\begin{equation}\label{kitten1}
\overline{R_{\lambda_D}\left(\overline{D(B(A))}\right)}= \overline{D'(B(A))}=R_{\lambda_D}\left(\overline{D(B(A))}\right),
\end{equation}
and so $\overline{D'(B(A))}$ is compact. Hence, $D'$ is a compact linear map.

To show part (3)  let $D$ be a weakly compact derivation.  Since $D$ is weakly compact, $\overline{D(B(A))}$ is 
weakly compact and so $R_{\lambda_D}\left(\overline{D(B(A))}\right)$ is 
weakly compact since bounded linear maps are weak-weak continuous. Thus equation (\ref{kitten}) holds with the 
closures taken in the weak topology, and so the weak closure of $D'(B(A))$ is weakly compact. 
Hence, $D'$ is a weakly compact linear map.

In each case we have a contradiction and so the result follows.
\end{proof}

\subsection{Compact derivations from $\ell^1(\Z^+)$}
In this section we look at compactness of derivations from  the semigroup algebra $\ell^1(\Z^+)$ - that is the Banach space $\ell^1(\Z^+)$ together with the product, 
\[
ab:=\left(\sum_{r=0}^na_rb_{n-r}:n\in\Z^+\right)_{n\in\Z^+},\quad\left(a=(a_n)_{n\in\Z^+}, b=(b_n)_{n\in\Z^+}\in\ell^1\right)
\]
- into its dual. It is standard (see for example \cite[2.1.13(v)]{Dales}), that this is a Banach algebra and that $c_{00}$ is dense in $A$. We identify $c_{00}$ 
with  the algebra $\C[t]$ of complex valued polynomials in one variable, so that the sequence $(0,1,0,\dots)=t$. It is standard that $\phi\mapsto(\phi(t^k))_{k\in\Z^+}$ is an isometric linear isomorphism 
from $A^*$ to $\ell^\infty$. The following proposition follows trivially from \cite[Lemma 3.3.1.]{Choi}. We provide a direct proof for the 
convenience of the reader.
\begin{prop}\label{1}
Let $\phi\in A^*$. The following are equivalent:
\begin{enumerate}
\item $(n\phi (t^{n-1}))_{n\in\N}\in\ell^\infty,$
\item $\phi=D(t)$ for some continuous derivation $D:A\rightarrow A^*.$
\end{enumerate}
Furthermore $\norm{(n\phi (t^{n-1}))_{n\in\Z^+}}_\infty= \norm{D}$.
\end{prop}
\begin{proof}
We first show that (1) implies (2) and that $\norm{(n\phi (t^{n-1}))_{n\in\Z^+}}_\infty\ge \norm{D}$.  Simple algebra yields that, for every $\phi \in A^*$, there is a unique derivation, $D$, from $\C[t]$ 
into $A^*$ with $\phi=D(t)$. By the derivation identity we have $D\left(t^k\right)\left(t^n\right)=\left(t^{k-1}\right)\cdot\phi \left(t^n\right)
=k\phi \left(t^{k+n-1}\right)$ and so, if we 
let $f$ be the polynomial $f=\sum_{k=0}^Na_kt^k$, we have, by linearity, that
\[
D(f)\left(t^n\right)= \sumo{k}{N}k a_k\phi \left(t^{(k+n-1)}\right).
\] 
Hence, since $\phi\mapsto\left(\phi\left(t^k\right)\right)_{n\in\N}$ is an isometric isomorphism,
\[
\norm{D(f)}=\sup_{n\in \N}\abs{\sumo{k}{N}k a_k\phi \left(t^{(k+n-1)}\right)}.
\]
For each $n\ge 0$,
\[
\left\vert  k\phi \left(t^{(k+n-1)}\right)\right\vert\le\left\vert (k+n) \phi (t^{(k+n-1)})\right\vert \le
 \norm{(n\phi (t^{n-1}))_{n\in\Z^+}}_\infty,
\] 
and so
\[
\norm{D(f)}\le \sup_{k,n\in N}\left\vert k \phi \left(t^{(k+n-1)}\right)\right\vert\sum_{k=0}^K\vert a_k \vert 
\le \norm{(n\phi (t^{n-1}))_{n\in\Z^+}}_\infty\norm{f}_1.
\]
Hence $D$ is bounded with norm at most 
$\norm{\left(n\phi \left(t^{n-1}\right)\right)}_\infty$ and so extends continuously a derivation 
$D:A\rightarrow A^*$  with 
$\norm D \le \norm{\left(n\phi \left(t^{n-1}	\right)\right)}_\infty$. 

To prove that (2) implies (1) and that $\norm D \ge \norm{\left(n\phi\left(t^{n-1}\right)\right)}_\infty$, note that 
\[
D(t^k)(1)=kt^{k-1}\cdot\phi(1)=k\phi (t^{k-1}),
\]
 and so 
\[
\abs{k\phi (t^{k-1})}=\abs{D(t^k)(1)}\le\norm D.
\]
Hence  $\norm D \ge \norm{(n\phi (t^{n-1})}_\infty$. The result follows.
\end{proof}
\begin{corollary}\label{2}The map 
\begin{eqnarray*} T&:&\mathcal D(A)\rightarrow A^*\\
&&D\mapsto D(\cdot)(1)
\end{eqnarray*} is an isometric isomorphism.
\end{corollary}
\begin{proof}
 By the derivation identity, $D(t^k)(1)=kD(t)(t^{k-1})$, and so 
\[
 \norm{D(\cdot)(1)}=\norm{(D(t^k)(1))_{k\in\Z^+}}_\infty=\norm{(kD(t)(t^{k-1}))_{k\in\Z^+}}_\infty,
 \]
which is equal to $\norm D$ by Proposition \ref{1}.
\end{proof}

\begin{thm}
A bounded derivation $D:A\rightarrow A^*$ is compact if and only if it has $(D(t^n)(1))_{n\in\N}\in c_{0}.$
\end{thm}
\begin{proof}
If $(D(t^n)(1))_{n\in\N}\in c_{0}$, then, by Corollary \ref {2}, we have that it is in the closure of the set 
$\{D:\left(D(t^n)(1)\right)_{n\in\N}\in c_{00}\},$
which consists of finite rank derivations. Hence $D$ is compact.
Now let $D:A\rightarrow A^*$ be a derivation such that $(D(t^n) (1))_{n\in\N}\in\ell^\infty\setminus c_{0}$. We shall show that the 
sequence $(D(t^k))_{k\in\N}$ 
has a subsequence with no convergent sub-subsequence. Without any loss of generality, we assume that $D$ has $\norm D=1$. 
 There exists $\eps>0$ and a 
sequence, $\seq{n}{k} \subseteq \N,$ such that, for all $n\in\N$, $n_k>n_{k-1}$ and  $\abs{D(t^{n_k})(1)}>\eps$. Let $k,l\in\N$. Then
\begin{eqnarray}
\nonumber \abs{D\left(t^k\right)\left(t^l\right)}&=&\abs{kt^{k+l-1}\cdot D(t)(1)}\\
\label{small}&=&\frac{k}{k+l}\abs{D\left(t^{k+l}\right)(1)}\le \frac{k}{k+l}.
\end{eqnarray}
Now suppose that $k+l\in \{n_k:k\in\N\}$. Then
\begin{eqnarray}
\nonumber \abs{D\left(t^k\right)\left(t^l\right)}&=&\abs{kt^{k+l-1}\cdot D(t)(1))}\\
\label{big}&=&\frac{k}{k+l}\abs{D\left(t^{k+l}\right)(1)}\ge \frac{\eps k}{k+l}.
\end{eqnarray}
Suppose that we have already chosen $j_1,\dots, j_{k-1}\in\N$ such that for all $i,i'\in\N$ with $i<i'<k$, we have $j_{i}<j_{i'}$ and
$\norm{D(t^{j_i})-D(t^{j_{i'}})}>\frac{\eps}{4}$. Choose $N\in \{n_k:k\in\N\}$ with $N>1000\eps^{-1}j_{k-1}$, and let
$l_k=\left\lfloor\frac{N}{2} \right\rfloor$ and $j_{k}=N-l$. Then, by (\ref{big}),
\[
\abs{D(t^{j_k})(t^{l_k})}\ge\frac{\eps j_k}{N}>\frac{\eps}{3}.
\]
Also,  if $m\le j_{k-1}$, then, by (\ref{small}),
\[
\abs{D\left(t^m\right)\left(t^{l_k}\right)}\le\frac{m}{m+l_k}\le\frac{j_{k-1}}{250\eps^{-1}j_{k-1}}<\frac{\eps}{250}.
\]
 Thus 
\[
\abs{D\left(t^{m}\right)\left(t^{l_k}\right)-D\left(t^{j_{k}}\right)\left(t^{l_k}\right)}>\frac{\eps}{4}.
\] 
In particular, if $i<k$, then $\norm{D(t^{j_i})-D(t^{j_{k}})}>\frac{\eps}{4}$. 
Hence, by induction, we obtain a sequence, $(j_i)_{i\in\N}$, such that, if $i,k\in \N$ and $i\ne k$ then 
$\norm{D(t^{j_i})-D(t^{j_{k}})}>\frac{\eps}{4}$. Thus $(D(t^{j_i}))_{i\in\N}$ has no convergent subsequence, and so, $D$ is not compact. 
\end{proof}
We conclude that the space of compact derivations on $A$ is linearly isomorphic to  $c_0$.

We finish with a relevant example due to J.F. Feinstein that appears in the present author's PhD thesis (\cite{MeThesis}.
\subsection{A non-compact, bounded derivation from a uniform algebra}
For a compact Hausdorff space $X$ let $C(X)$ be the algebra of continuous functions from $X$ to $\C$ equipped with the uniform norm, which we denote by $\abs\cdot_X$. For a compact subset, $X$, of the complex plane we let $R_0(X)$ be the algebra of rational functions with no poles contained in $X$. We let $R(X)$ be closure of $R_0(X)$ in $C(X)$.  
Let $\Delta$ be the closed unit disc. We shall construct a plane set $X$ by removing a sequence, $(D_n)_{n\in\N}$, of open discs  from $\Delta$  such that there is a non-compact, bounded 
derivation from $R(X)$ into a symmetric Banach $R(X)$-bimodule. We shall need the 
following result, which is \cite[Lemma 3]{FeinsteinTrivJen}.
\begin{prop}\label{ftj}
Let $\Delta$ be the closed unit disc and $(D_n)_{n\in\N}$ a sequence of open discs each contained in $\Delta$. Set 
\[X:=\Delta\setminus \bigcup_{i=1}^\infty D_n.\] 
We set $r_n=r(D_n)$ and  for each 
$z\in X$ we set $s_n(z)=\up{dist}(z,D_n)$. We also set $r_0= 1$ and $s_0(z)=1-\abs z$. If $s_n(z)>0$ for all $n\in\N$ then for
$f\in R_0(X)$  we have 
\[
\abs{f'(z)}\le\sum_{j=0}^\infty \frac{r_j}{s_j(z)^2}\abs f_{X}.
\]
\end{prop}
\begin{ex}
 Let $I=\left[0,\frac{1}{2}\right]$. For any compact plane set $X$ with $I\subseteq X$, we make 
$C(I)$ a symmetric Banach $R(X)$-bimodule by defining the action 
\[(f\cdot g)(x)=(g\cdot f)(x)= f(x)g(x)\quad f\in R(X),\,g\in C(I).\] 
It is clear that the map 
$D:R_0(X)\rightarrow C(I)$ given by $D(f)=f'|_I$ is a derivation. We shall construct a collection $\{D_n:n\in\N\}$ of disjoint open discs contained in $\Delta$ such that, setting $X=\Delta\setminus\bigcup_{n=1}^\infty D_n$, we have
$I\subseteq X$ and such that  the derivation $D$ is  bounded and so extends by continuity to a bounded derivation $R(X)\rightarrow C(I)$ which is not compact. For $n\in\N$, let 
$I_n=\left[\frac{1}{2}-2^{-n},\frac{1}{2}-2^{-(n+1)}\right[\,$ and $x_n=\frac{1}{2}-3\cdot2^{-(n+1)}$; that is, $x_n$ is the 
midpoint of $I_n$. Choose $y_n\in\lopen0,1\ropen$ small enough that 
\begin{eqnarray}
\label{yn1}x_n+iy_n,&\in& \Delta\\
\label{yn2}\frac{1}{(1-y_n)^2}&<&2, \textrm{ and}\\
\label{yn3}\frac{y_n^2}{\left(\left(\left(2^{-2(n+1)}\right)-y_n^2\right)^{\frac{1}{2}}+y_n^2\right)^2}&<&2^{-(n+1)}.
\end{eqnarray}
Set $a_n=x_n+iy_n$, $r_n=y_n^2$, $D_n=B(a_n,r_n)$ and $X=\Delta\setminus\bigcup_{n=1}^\infty D_n$.  We also set $r_0=1$. 
Let $z\in X$. We let $s_0(z)=1-\abs z$ and for $n\in\N$, let $s_n(z)=\up{dist}(z,D_n)$. 
Now let $x\in I$. Then $s_0(x)\ge\frac{1}{2}$ so $\frac{r_0}{s_0(x)^2}=s_0(x)^{-2}\le 4$. Also, either: $x=\frac{1}{2}$, in which case, for each $j\in\N$, 
$s_j(z)\ge\up{dist}(D_j,\R\-I_j)=\left(2^{-2(j+1)}+y_j^2\right)^{\frac{1}{2}}-y_j^2$; or there exists a unique $n\in\N$ such that 
$x\in I_n$. In this second case, for $j\in \N$, 
\begin{equation}
\label{esti}s_j(x)\ge\left\{\begin{array}{ll}
 \up{dist}(D_n,\R)=y_n-r_n=y_n-y_n^2&\textrm{ if }j=n\\
\up{dist}(D_j,\R\-I_j)\ge\left(2^{-2(n+1)}+y_n^2\right)^{\frac{1}{2}}-y_n^2&\textrm{ if }j\ne n
\end{array}\right..
\end{equation}
Thus, by (\ref{yn2}), (\ref{yn3}) and (\ref{esti}), 
\begin{eqnarray*}
 \sum_{j=0}^\infty \frac{r_j}{s_j(x)^2}&\le& 4+\frac{y_n^2}{(y_n-y_n^2)^2}+
\sum_{j=1}^\infty\frac{y_j^2}{\left(\left(2^{-2(j+1)}+y_j^2\right)^{\frac{1}{2}}-y_j^2\right)^2}\\
&<& 4+2+\sum_{j=1}^\infty2^{-(j+1)}=\frac{13}{2}.
\end{eqnarray*}
By Proposition \ref{ftj}, this implies that  $\abs{f'}_I<\frac{13}{2}\abs{f}_{X}$, for $f\in R_0(X)$. Hence $D$ is a bounded derivation from $R_0(X)$ 
to $C(I)$. We extend $D$ by continuity to a derivation from $R(X)$ to $C(I)$, which we shall also call $D$. It remains to show that $D$ is not compact. Let $n\in N$, and let
\[
f_n(z)=\frac{r_n}{z-a_n},\quad(z\in X).
\]
Then $\abs{f_n}_{X}=1$. Also 
\[
f_n'(z)=\frac{-r_n}{(z-a_n)^2},\quad(z\in X).
\]
Clearly, for each $x\in[0,1/2\ropen$, $f_n'(x)\rightarrow 0$ as $n\rightarrow\infty$. Thus, if $\left(f_n'|_I\right)_{n\in\N}$ were to have a convergent subsequence the limit would have to 
be the zero function. However, $\abs{f_n'(x_n)}=1$ for each $n\in\N$. Hence $(D(f_n))_{n\in\N}=\left(f_n'|_I\right)_{n\in\N}$ has no convergent subsequence, and thus $D$ 
is not a compact linear map. 
\end{ex}


\section*{Acknowledgements}
This work is adapted from the author's PhD thesis (\cite{MeThesis}), which was produced under the supervision of J. F. Feinstein and with the support of a grant from the EPSRC (UK). 

This paper is based on a lecture delivered at the 19$^\mathrm{th}$
International Conference on Banach Algebras held at B\k{e}dlewo, July 14--24,
2009. The support for the meeting by the Polish Academy of Sciences, the
European Science Foundation under the ESF-EMS-ERCOM partnership, and the
Faculty of Mathematics and Computer Science of the Adam Mickiewicz University
at Pozna\'n is gratefully acknowledged. The author is grateful to FCT (Portugal) for paying his travel costs to this conference and funding his continuing research under the grant SFRH/BPD/40762/2007.

\def\cprime{$'$}
\providecommand{\bysame}{\leavevmode\hbox to3em{\hrulefill}\thinspace}
\providecommand{\MR}{\relax\ifhmode\unskip\space\fi MR }
\providecommand{\MRhref}[2]{%
  \href{http://www.ams.org/mathscinet-getitem?mr=#1}{#2}
}
\providecommand{\href}[2]{#2}

Departamento de Matem\'atica,
Instituto Superior T\'ecnico, \\
Av. Rovisco Pais, 
1049-001 Lisboa,
Portugal \\
E-mail: mheath@math.ist.utl.pt
\end{document}